\documentclass[a4paper]{article}

\usepackage[english]{babel}
\usepackage[utf8x]{inputenc}
\usepackage[T1]{fontenc}

\usepackage[a4paper,top=3cm,bottom=2cm,left=3cm,right=3cm,marginparwidth=1.75cm]{geometry}

\usepackage{amsmath, amsfonts, amsthm, amssymb, MnSymbol}
\usepackage{graphicx}
\usepackage[colorinlistoftodos]{todonotes}
\usepackage[colorlinks=true, allcolors=blue]{hyperref}
\usepackage{accents}
\usepackage[margin=1cm]{caption}
\usepackage{ulem}

\usepackage{amsmath}
\usepackage{amsthm}
\usepackage{amsfonts}
\usepackage{amssymb}
\usepackage{empheq}
\usepackage[siunitx]{circuitikz}
\usepackage{tikz}
\usepackage{fancybox}
\usepackage{epsfig}
\usepackage{soul}
\usepackage[framemethod=tikz]{mdframed}

\newcommand\e{\varepsilon}

\newcommand\R{\mathbb{R}}

\newcommand\N{\mathbb{N}}
\newcommand\G{\mathcal{G}}
\renewcommand\P{\mathbb{P}}
\newcommand\per{\text{per}}
\newcommand\E{\mathbb{E}}

\renewcommand\S{\mathcal{S}}

\theoremstyle{definition}
\newtheorem{theorem}{Theorem}
\newtheorem{definition}[theorem]{Definition}

\newtheorem{corollary}[theorem]{Corollary}
\newtheorem{proposition}[theorem]{Proposition}

\newcounter{rcnt}[section]

\begin{document}

\renewcommand{\abstractname}{Abstract}

\begin{center}
\huge{Permanental Graphs} \\
\vspace{1em}
\large{Daniel Xiang} \hspace{1em} {Peter McCullagh} \\
\vspace{1em}
{Department of Statistics, University of Chicago}\\
\vspace{1em}
\today \\
\vspace{3em}
\end{center}

\begin{abstract}
The two components for infinite exchangeability of a sequence of distributions $(P_n)$ are (i) consistency, and (ii) finite exchangeability for each $n$. A consequence of the Aldous-Hoover theorem is that any node-exchangeable, subselection-consistent sequence of distributions that describes a randomly evolving network yields a sequence of random graphs whose expected number of edges grows quadratically in the number of nodes. In this note, another notion of consistency is considered, namely, delete-and-repair consistency; it is motivated by the sense in which infinitely exchangeable permutations defined by the Chinese restaurant process (CRP) are consistent. A goal is to exploit delete-and-repair consistency to obtain a nontrivial sequence of distributions on graphs $(P_n)$ that is sparse, exchangeable, and consistent with respect to delete-and-repair, a well known example being the Ewens permutations \cite{tavare}. A generalization of the CRP$(\alpha)$ as a distribution on a directed graph using the $\alpha$-weighted permanent is presented along with the corresponding normalization constant and degree distribution; it is dubbed the Permanental Graph Model (PGM). A negative result is obtained: no setting of parameters in the PGM allows for a consistent sequence $(P_n)$ in the sense of either subselection or delete-and-repair.

\end{abstract}

\section{Introduction}

The two components for infinite exchangeability of a sequence of distributions $(P_n)$ are (i) consistency, and (ii) finite exchangeability for each $n$. From a modeling perspective, exchangeability is an assumption that is natural in a setting where the statistical units are labelled in an arbitrary manner. If the process being studied is a record of the relationships $(i,j) \mapsto X_{ij}$ between ordered pairs of units $(i, j)$, the process $X^\sigma$ after label permutation has components $(i,j) \mapsto X_{\sigma(i), \sigma(j)}$. As a matrix, $X^\sigma = \sigma X \sigma^{-1}$ is obtained from $X$ by permuting rows and columns, i.e., by conjuation by $\sigma\in \S_n$. In this setting, finite exchangeability means that all of the permuted matrices have the same joint distribution.

Subselection consistency is not specific to Boolean matrices, but applies to real-valued matrices and to more general arrays. It requires that for $X \in \{0,1\}^{n\times n}$ distributed according to $P_n$, the top left $(n-1) \times (n-1)$ submatrix of $X$ is distributed according to $P_{n-1}$.
The operation defined by deleting the last row and column of the adjacency matrix does not rely on the fact that the entries of an adjacency matrix are boolean valued. Sampling a sub-network according to subselection amounts to picking a subset of vertices and including only the edges between pairs of vertices in the selected subset. The following calculation, replicated from \cite{networksbook}, demonstrates how edge sparsity, node-exchangeability, and subselection-sampling are at odds with each other.

\begin{proof}[Calculation]
$X\in \{0,1\}^{n\times n}$ is a simple directed graph possibly containing self-loops. $X$ is assumed to be ``sparse'', i.e.
\begin{align*}
\sum_{i,j} X_{ij} = \e n = o(n^2),
\end{align*}
for some $\e > 0$ independent of $n$. Let $\sigma \in \S_n$ be drawn uniformly at random, and put $Y \doteq \sigma X \sigma^{-1}$. In this model, we observe the top left $m\times m$ sub-matrix of $Y$, denoted $Y_{1:m,1:m}$, which is exchangeable according to this construction. Assume $m \ll n$, at the order $m = o(\sqrt{n})$. By the union bound,
\begin{align*}
\P\left(\bigcup_{i,j\leq m} \{Y_{ij} = 1\} \right) &\leq \sum_{i,j\leq m} \P(Y_{ij}=1) \\
&\asymp m^2 \P(Y_{12}=1). \tag{exchangeability}
\end{align*}
The event $Y_{12}=1$ corresponds to having picked two vertices uniformly at random from the $\binom{n}{2} \asymp n^2$ possible pairs, and observing an edge between them. Hence
\begin{align*}
\P(Y_{12}=1) \asymp \frac{\e n }{n^2} = \frac{\e}{n}.
\end{align*}
Plugging this back into the union bound, we find that
\begin{align*}
\P\left(\bigcup_{i,j\leq m} \{Y_{ij} = 1\} \right) \leq \frac{\e m^2}{n} \approx 0. \tag{$m=o(\sqrt{n})$}
\end{align*}
The network we ``observe'' contains no edges with high probability. Put more plainly, we observe no network at all!

\end{proof}


In order to resolve the contradiction suggested by the calculation above, at least one of sparsity, node-exchangeability, or subselection-consistency must be modified. In this note, an alternative notion of consistency is considered, namely, delete-and-repair consistency. It is motivated by the sense in which infinitely exchangeable permutations defined via the Chinese restaurant process (CRP) are consistent.


%

The $\alpha$-weighted permanent is used to generalize the probability function associated to the partitions and permutations generated from the one parameter CRP$(\alpha)$ to probabilities on general directed graphs. Prescribing a probability to a directed graph according to the $\alpha$ permanent of its adjacency matrix automatically yields an exchangeable distribution with tractable calculations for the normalization constant and degree distribution similar to that of the Erd\"{o}s-R\'{e}nyi$(n,p)$ model. The negative result we obtain is that any setting of the parameters for the permanental graphs allows for neither delete-and-repair nor subselection consistency. All proofs are deferred to Section \ref{calculations}.

\section{Permanental Graphs}

The $\alpha$-weighted matrix permanent $\per_\alpha : \R^{n\times n} \to \R$ is a matrix functional defined by,
\begin{align*}
\per_\alpha(G) \doteq \sum_{\sigma \in \S_n} \alpha^{\# \sigma} \prod_{i=1}^n G_{i,\sigma(i)},
\end{align*}
where the sum runs over all permutations $\sigma : [n]\to[n]$, and $\#\sigma$ is the number of cycles. If $G$ is Boolean,  the product $\prod G_{i, \sigma(i)}$ is equal to one if $\sigma$ is contained as a sub-graph in~$G$, and zero otherwise. Thus $\per_1(G)$~is the number of permutations contained as sub-graphs in~$G$, and $\per_\alpha(G)$ is the cycle-weighted count.
The matrix permanent is recovered by setting $\alpha=1$, whereas the determinant is obtained as $\per_{-1}(G) = (-1)^n \det(G)$. In this note, the word graph or $n$-graph means a simple directed graph, with no multiple edges, but possibly containing self-loops. In other words, each $n$-graph is a Boolean matrix of order $n$. 

For each $n\ge 1$, let $\G_n \subset \{0,1\}^{n \times n}$ be any subset of $n$-graphs satisfying the following conditions:
\begin{enumerate}
\item $\G_n$ is closed under conjugation:
\begin{align*}
\sigma \G_n \sigma^{-1} = \G_n, \hspace{2em} \sigma\in\S_n
\end{align*}
\item There exists $G\in \G_n$ and $\sigma\in\S_n$ such that $\sigma \subset G$.
\end{enumerate}
Examples that we have in mind include the whole space, $\G_n = \{0,1\}^{n\times n}$, the permutations $\G_n=\S_n$, permutations having no fixed points for $n\ge 2$, single-cycle permutations, equivalence relations or set partitions as graphs, and so on. Condition 1 means that $\G_n$ is a union of group orbits, while condition 2 excludes trivialities such as graphs having fewer edges than vertices.

Consider the graph distribution
\begin{align}
\label{per}
P_n(G) \propto \textbf{1}_{\{G \in \G_n\}}\per_\alpha(G) = \textbf{1}_{\{G \in \G_n\}} \sum_{\sigma \in \S_n} \alpha^{\# \sigma} \textbf{1}_{\{\sigma \subset G\}},
\end{align}
that is proportional to the $\alpha$-permanent restricted to $\G_n$. Condition 2 gives $\sum_{G\in\G_n} \per_\alpha(G) >0$ when $\alpha > 0$, so the normalizing constant is strictly positive.
Note that $P_n$ is automatically exchangeable, because for any $\tau \in \S_n$,
\begin{align*}
P_n(G^\tau) &\propto \sum_{\sigma \in \S_n} \alpha^{\#\sigma} \prod_{i=1}^n G_{\tau(i),\tau(\sigma(i))} \\
&= \sum_{\sigma \in \S_n} \alpha^{\# \sigma} \prod_{j=1}^{n}G_{j,\tau\sigma\tau^{-1}(j)} \tag{$i = \tau^{-1}(j)$ for some $j$} \\
&= \sum_{\sigma \in \S_n} \alpha^{\# \tau\sigma\tau^{-1}}\prod_{j=1}^{n}G_{j,\tau\sigma\tau^{-1}(j)} \tag{$\# \sigma = \# \tau \sigma\tau^{-1}$} \\
&= \sum_{\sigma \in \S_n} \alpha^{\#\sigma} \prod_{i=1}^n G_{i,\sigma(i)} \tag{sum ranges over all $\sigma \in \S_n$}.
\end{align*}
Consistency in any sense is not immediately clear. When $\G_n$ is taken to be the set of adjacency matrices corresponding to partitions of $[n]$, we have
\begin{align*}
P_n(\pi) \propto \sum_{\sigma \in \S_n} \alpha^{\# \sigma} \textbf{1}_{\{\sigma \subset \pi\}},
\end{align*}
where $\sigma \subset \pi$ means the graph induced by $\sigma$ is a subgraph of the graph induced by $\pi$, i.e. the cycles of the permutation $\sigma$ coincide with the blocks of the partition $\pi$. Further simplification gives
\begin{align*}
P_n(\pi) \propto \alpha^{\# \pi} \cdot \# \{\sigma \in \S_n : \sigma \subset \pi\} = \alpha^{\# \pi}\prod_{j=1}^{\# \pi} (n_j-1)!,
\end{align*}
where $n_j$ are the block sizes of $\pi$, and $\# \pi$ is the number of blocks in $\pi$. It follows from the above formula that the sequence $(P_n)$ coincides with the CRP$(\alpha)$ on partitions when $\G_n$ is taken to be the set of adjacency matrices corresponding to partitions. When $\G_n$ is the set of permutations, similar reasoning shows that $(P_n)$ is the same as the CRP$(\alpha)$ for permutations. For an introduction to the CRP for partitions and permutations, see Section 3.1 of \cite{pitman}. Letting $\G_n = \{0,1\}^{n\times n}$ have unrestricted support, and including an additional ``odds'' parameter $\beta > 0$, the following collection of distributions, called the Permanental Graph Model, is obtained.
\begin{theorem}[Permanental Graph Model]
\label{dgm}
Let $\G_n = \{0,1\}^{n \times n}$ be the whole space, and put
\begin{align}
\label{permgraph}
P_n(G) \propto \beta^{\# G} \per_\alpha(G),
\end{align}
for $G \in \G_n$, $\alpha,\beta > 0$, where $\# G = \sum_{i,j\leq n} G_{ij}$ is the number of edges in $G$. Then the normalization constant is
\begin{align*}
z_n(\alpha,\beta) \doteq \sum_{G \in \G_n} \beta^{\# G} \per_\alpha(G) = \alpha_{n\uparrow 1}\left( \frac{\beta}{1+\beta} \right)^n (1+\beta)^{n^2},
\end{align*}
where $\alpha_{n\uparrow 1} \doteq \alpha (\alpha+1)\cdots (\alpha+n-1)$ is the rising factorial starting at $\alpha$. The degree distribution is given by
\begin{align*}
\sum_{j=1}^n G_{1j} - 1 \sim \text{Binom}\left(n-1,\frac{\beta}{1+\beta}\right).
\end{align*}
\end{theorem}
A consequence of the above theorem is that the expected number of edges in this model grows as 
\begin{align}
\label{expedges}
\E\left[\sum_{i,j} G_{ij} \right]\sim \frac{\beta n^2}{1+\beta},
\end{align}
in the sense of $a_n \sim b_n \iff a_n/b_n \to 1$.

\subsection{Two Notions of Consistency}

Before stating the negative result, we present two notions of projection from $(n+1)$-graphs to $n$-graphs.

\begin{definition}
The subselection map $\varphi^{\text{ss}}_n : \{0,1\}^{(n+1)\times(n+1)} \to \{0,1\}^{n\times n}$ is defined by
\begin{align*}
\left(\varphi^{\text{ss}}_n (G)\right)_{ij} = G_{ij}, \hspace{2em} i,j \in [n].
\end{align*}
\end{definition}

\begin{definition}
The delete-and-repair map $\varphi^{\text{dr}}_n : \{0,1\}^{(n+1)\times (n+1)} \to \{0,1\}^{n\times n}$ is defined by
\begin{align}
\label{dr}
\left(\varphi^{\text{dr}}_n (G)\right)_{ij} = G_{ij} \vee (G_{in} \wedge G_{nj}), \hspace{2em} i,j \in [n],
\end{align}
where $\vee$ and $\wedge$ represent boolean ``or'' and ``and'' respectively.
\end{definition}
Note that the definition of the delete-and-repair projection mapping is specific to matrices with boolean valued entries, whereas the definition of the subselection projection mapping applies equally well to matrices whose entries are real valued. In words, given a graph on $n+1$ vertices, the delete-and-repair projection (\ref{dr}) deletes all edges connecting to node $n+1$, and repairs edges for pairs of nodes (including self pairs, $(v,v)$) between which there was a length 2 path going through node $n+1$. These notions of projection are illustrated in Figure \ref{projexamples}.

\begin{figure}[h]
\centering
\includegraphics[width=80mm,height=35mm]{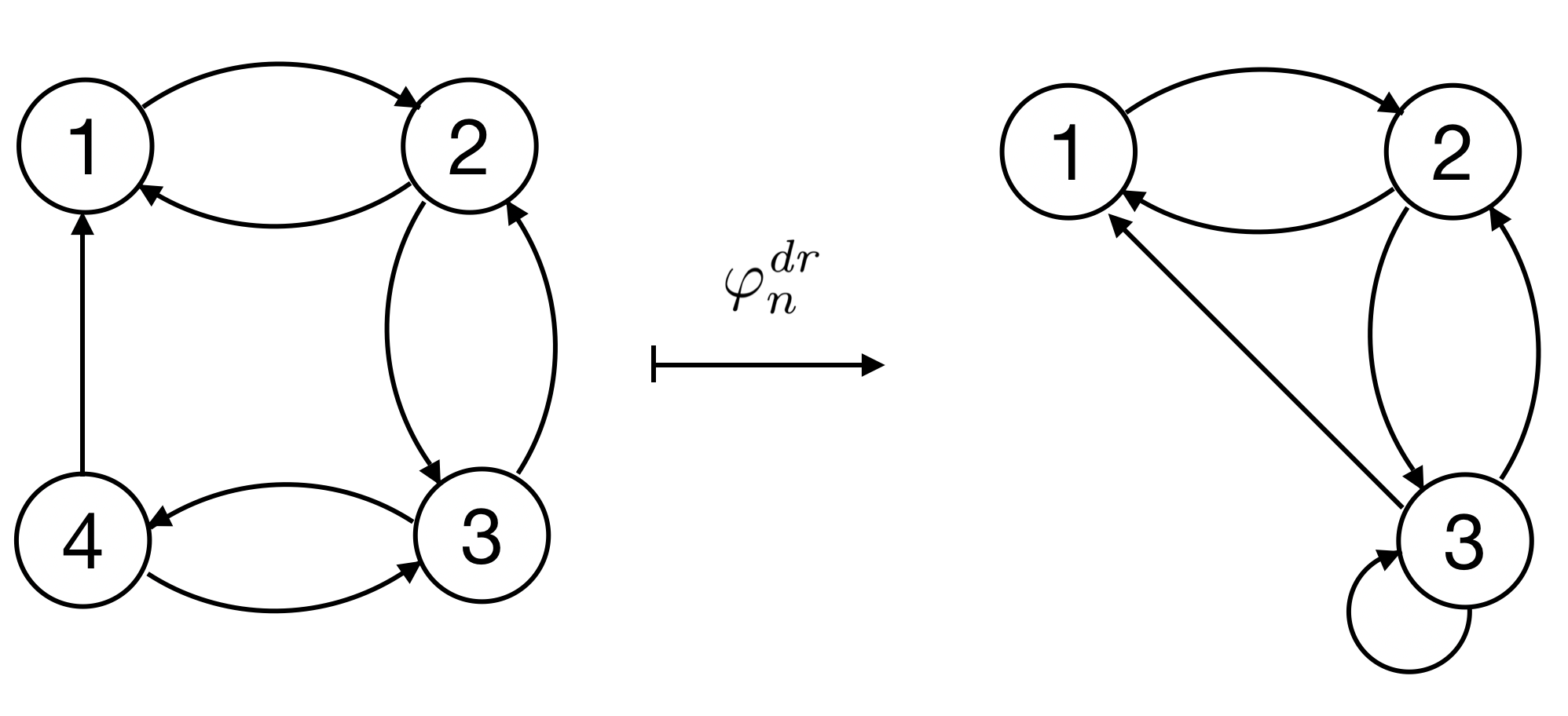}
\includegraphics[width=80mm,height=35mm]{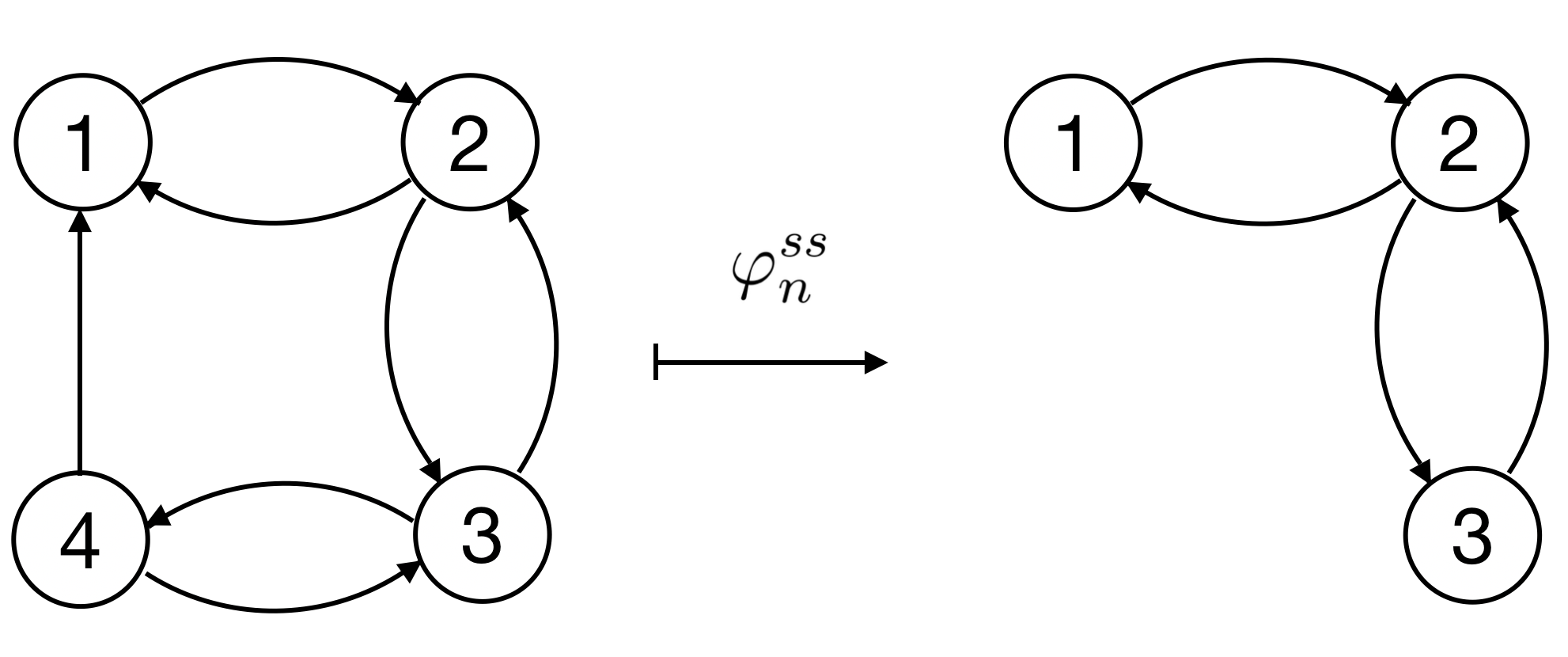}
\caption{Pictured above (top) is an example of a directed graph on vertex set $[4]$ projected down to a directed graph on $[3]$ according to the delete-and-repair operation (\ref{dr}), and (bottom) the same graph projected down according to subselection.}
\label{projexamples}
\end{figure}

The CRP($\alpha$) for partitions is recovered from (\ref{per}) by setting $\G_n$ equal to the set of partition matrices, while the CRP$(\alpha$) for permutations is recovered by setting $\G_n$ equal to the set of permutation matrices. It is straightforward to check that CRP$(\alpha$) on partitions is consistent with respect to both subselection and delete-and-repair, which in this case are equivalent due to the transitivity property of equivalence relations. However, when viewed as a distribution on permutations, the CRP$(\alpha)$ is consistent only with respect to delete-and-repair. 

\begin{corollary}
\label{crpperm}
Put $\G_n = \{\Pi_\sigma \in \{0,1\}^{n \times n} : \sigma \in \S_n\}$, where $(\Pi_\sigma)_{ij}=1 \iff \sigma(i)=j$, and $\beta=1$ in (\ref{permgraph}). Then the following probabilities,
\begin{align}
\label{ewens}
P_n(\sigma) = \frac{\alpha^{\# \sigma}}{\alpha_{n\uparrow 1}},\hspace{2em}\text{$\sigma \in \S_n$},
\end{align}
define a valid probability distribution on $\S_n$ that is delete-and-repair consistent. Here, $\alpha_{n\uparrow 1} \doteq \alpha(\alpha+1)\cdots(\alpha+n-1)$ is the rising factorial starting at $\alpha$. The above distribution is known as the $\text{CRP}(\alpha)$ for permutations.
\end{corollary}

From Corollary \ref{crpperm}, it follows that sparsity and node-exchangeability are not mutually exclusive properties. Indeed, the Ewens permutations described by (\ref{ewens}) are delete-and-repair consistent, node-exchangeable, and sparse, as they contain exactly $n = o(n^2)$ edges for each $n$.

To see why the CRP($\alpha$) for permutations is not subselection consistent, consider the permutation $(123)$. It has adjacency matrix satisfying,
\begin{align*}
\begin{bmatrix}
0&1&0\\0&0&1\\1&0&0
\end{bmatrix} \stackrel{\varphi^{\text{ss}}_n}{\longmapsto} \begin{bmatrix}
0&1\\0&0
\end{bmatrix}
\end{align*}
under subselection of the first two vertices. A permutation matrix must have a single ``1'' in each row and column, so the adjacency matrix on the right does not correspond to a permutation. Thus, these distributions cannot be consistent in the sense of subselection. However, the distribution can be specified by a generative a process (CRP seating plan), meaning that the law of total probability is satisfied; the distributions are consistent in some sense. Indeed, a more natural notion of projection in this example is delete-and-repair, for which we would instead obtain,
\begin{align*}
\begin{bmatrix}
0&1&0\\0&0&1\\1&0&0
\end{bmatrix} \stackrel{\varphi^{\text{dr}}_n}{\longmapsto} \begin{bmatrix}
0&1\\1&0
\end{bmatrix},
\end{align*}
according to the formula (\ref{dr}). For a permutation $\sigma \in \S_{n+1}$, the delete-and-repair operation deletes node $n+1$ from its cycle, while repairing an edge from the preimage of $n+1$ to the image of $n+1$ under $\sigma$. If $n+1$ is contained in its own cycle, then the node and the cycle are entirely removed, and no edges are repaired.

\subsection{A Negative Result}

As a graph is projected down to a smaller graph according to the boolean operation (\ref{dr}), edges may be repaired and thus added to the graph. It follows from this observation and (\ref{expedges}) that, for the distributions (\ref{permgraph}) to be delete and repair consistent, it is natural to expect the $\beta$ parameter to be decreasing in $n$. One may then suspect that (\ref{expedges}) is $o(n^2)$ for a delete and repair consistent sequence $(P_n)$. But, the next result states that the PGM$(\alpha_n,\beta_n)$ defined in (\ref{permgraph}), which has unrestricted support, i.e. $\G_n = \{0,1\}^{n\times n}$, does not admit a consistent sequence $(P_n)$ in either sense described above, for any sequence of pairs $(\alpha_n,\beta_n)$.

\begin{proposition}
\label{imposs}
For no sequence of pairs $(\alpha_n,\beta_n)_{n\in \N} > 0$ are the distributions
\begin{align*}
P_n(G) \propto \beta_n^{\# G} \per_{\alpha_n} (G), \hspace{2em} G \in \G_n \doteq \{0,1\}^{n\times n}
\end{align*}
delete-and-repair or subselection consistent. Equivalently, the statement
\begin{align}
\label{ltp}
P_n(G) = \sum_{G' \in \G_{n+1}: \varphi^\bullet_n(G') = G}  P_{n+1}(G'),\hspace{2em} G \in \G_n,n \in \N
\end{align}
is not true, where $\bullet$ is either dr (delete-and-repair) or ss (subselection).

\end{proposition}

A sketch of the proof of Proposition \ref{imposs} is provided below, while the full proof is presented in Section \ref{calculations}.

\begin{proof}[Proof sketch.]
When $\bullet$ is dr, the equation in (\ref{ltp}) can be written
\begin{align}
\label{norm}
\frac{z_{n+1}(\alpha_{n+1},\beta_{n+1})}{z_n(\alpha_n,\beta_n)} = \frac{\sum_{\sigma'\in \S_{n+1}}\alpha_{n+1}^{\sigma'} \sum_{G' : \varphi^\text{dr}_n(G')=G}\beta_{n+1}^{\# G'}\textbf{1}_{\{\sigma'\subset G'\}}}{\beta_n^{\# G} \sum_{\sigma \in \S_n} \alpha_n^{\# \sigma} \textbf{1}_{\{\sigma \subset G\}}}.
\end{align}
One consequence is that the right hand side is constant in $G \in \{0,1\}^{4 \times 4}$. The following two graphs,
\begin{align*}
G_1 = \begin{bmatrix}
0&1&0&0\\1&0&1&0\\0&1&0&1\\1&0&1&0
\end{bmatrix} ,\hspace{1em} G_2 = \begin{bmatrix}
0&1&0&0\\0&0&1&0\\1&1&0&1\\1&0&0&1
\end{bmatrix}
\end{align*}
have the same number of edges. $G_1$ contains the permutations $(1234)$ and $(12)(34)$. $G_2$ contains the permutations $(1234)$ and $(123)(4)$. These graphs are visualized in Figure \ref{g1g2}. Hence the denominators in (\ref{norm}) are equal,
\begin{align*}
\beta_4^{\# G_1} \sum_{\sigma \in \S_4} \alpha_4^{\# \sigma}  \textbf{1}_{\{\sigma \subset G_1\}} = \beta_4^{\# G_2} \sum_{\sigma \in \S_4} \alpha_4^{\# \sigma}  \textbf{1}_{\{\sigma \subset G_2\}}.
\end{align*}
The key observation used to exhibit a contradiction is that the numerator of the right hand side of (\ref{norm}) depends on the set of graphs which project down to $G$ according to the delete and repair operation (\ref{dr}). It is shown in Section \ref{calculations} that the set of graphs in $\G_5$ which project down to $G_2$ has greater cardinality than the corresponding set of graphs for $G_1$. Upon computing the right hand side of (\ref{norm}) for $G_1$ and $G_2$, it becomes clear that they are not equal for any pair $(\alpha_4,\beta_4),(\alpha_5,\beta_5) > 0$, hence contradicting the statement (\ref{ltp}). This computation, along with the proof for subselection inconsistency, can be found in Section \ref{calculations}.

\end{proof}

\begin{figure}[h]
\centering
\includegraphics[width=80mm,height=40mm]{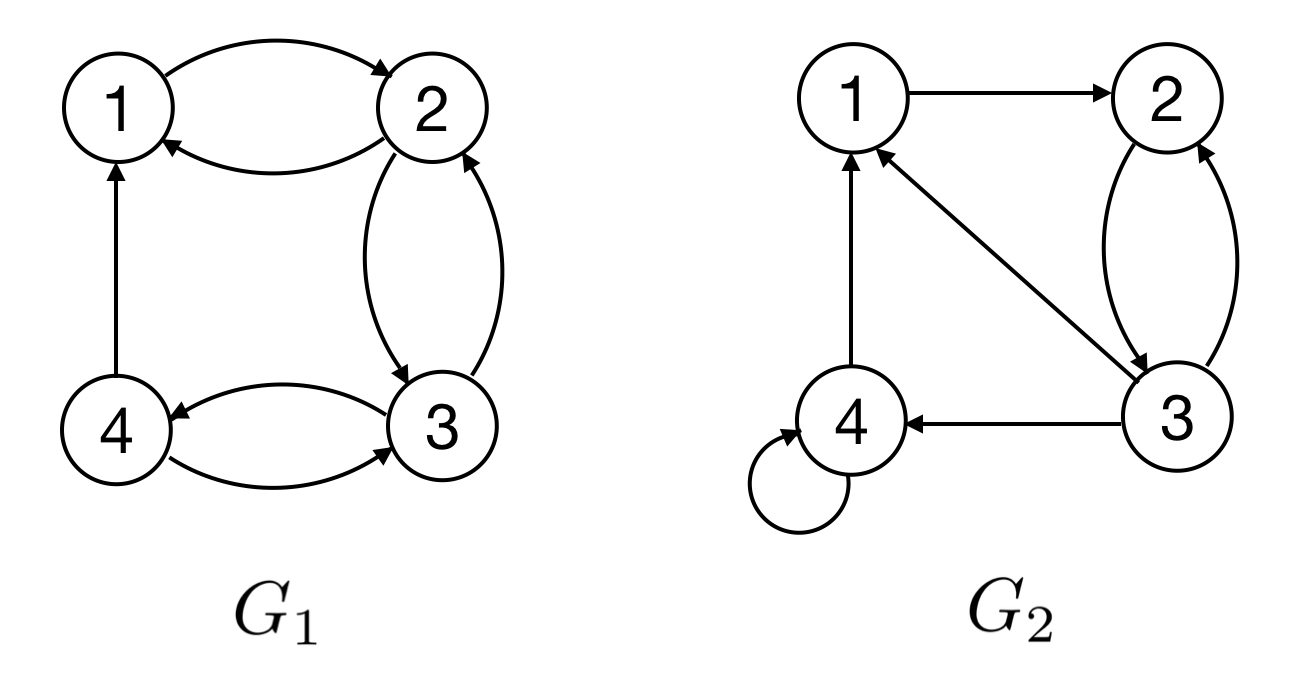}
\caption{Graphs $G_1 = (1234) \vee(12)(34) \vee(23)$ and $G_2 = (1234)\vee (123)(4)\vee(23)$ for which the right hand sides of (\ref{norm}) are not the equal.}
\label{g1g2}
\end{figure}

Note that the above proposition implies that no constant (in $n$) pair $(\alpha,\beta)$ allows for a consistent sequence of distributions $(P_n)$ in either sense discussed above.

\section{Conclusion}

We have investigated a collection of exchangeable distributions on graphs, defined via the $\alpha$ permanent. Setting $\G_n$ to be the set of all directed graphs allows for tractable calculations regarding the normalizing constant and degree distribution. A negative result was obtained; no choice of the parameters $(\alpha_n,\beta_n)$ yields a consistent collection of graph distributions in the sense of subselection or delete-and-repair. The following questions remain unaddressed: Can we distinguish between the permanental graph model and the Erd\"{o}s-R\'{e}nyi$(n,\frac{\beta}{1+\beta})$ graph as $n \to \infty$? Can we find a different delete-and-repair projective system $(\G_n)_{n \in \N}$ (meaning $\varphi^\text{dr}_n(\G_{n+1}) \subset \G_n$) for which similar calculations are tractable? It would be surprising if the only delete-and-repair consistent graphs were exchangeable partitions and permutations.

\newpage

\section{Proofs}
\label{calculations}

\begin{proof}[Proof of Theorem \ref{dgm}]
The sum over all graphs $G \in \G_n$ is
\begin{align*}
z_n(\alpha,\beta) &\doteq \sum_{G \in \G_n} \beta^{\# G}\sum_{\sigma} \alpha^{\# \sigma} \prod_{i=1}^n G_{i,\sigma(i)} \\
&= \sum_{\sigma} \alpha^{\# \sigma} \sum_{G \in \G_n} \beta^{\# G} \textbf{1}_{\{\sigma \subset G \}} \\
&= \underbrace{\sum_{\sigma} \alpha^{\# \sigma}}_{\alpha_{n\uparrow 1}} \sum_{m=n}^{n^2} \beta^m \underbrace{\sum_{G \in \G_n : \# G = m} \textbf{1}_{\{\sigma \subset G\}}}_{=\binom{n^2-n}{m-n}}\\
&= \alpha_{n\uparrow 1} \beta^n \sum_{m=0}^{n^2-n} \beta^{m} \binom{n^2-n}{m} \\
&= \alpha_{n\uparrow 1} \beta^n (1+\beta)^{n^2-n},
\end{align*}
where the first underbrace in the third line is due to Lemma \ref{crpperm}, and last equality is by the binomial formula. Hence the normalization constant is
\begin{align*}
z_n(\alpha,\beta) = \alpha_{n\uparrow 1} \left(\frac{\beta}{1+\beta}\right)^n (1+\beta)^{n^2}.
\end{align*}
Put $G_{1\bullet} \doteq \sum_{j=1}^n G_{1j}$, whose distribution is computed below. 
\begin{align*}
\P(G_{1\bullet} = k+1) &= \sum_{G \in \G_n: G_{1\bullet}=k+1} P_n(G) \\
&= \frac{1}{z_n(\alpha,\beta)} \sum_{G \in \G_n: G_{1\bullet}=k+1} \beta^{\# G}\sum_{\sigma}\alpha^{\# \sigma} \textbf{1}_{\{\sigma \subset G\}} \\
&= \frac{1}{z_n(\alpha,\beta)} \sum_{\sigma}\alpha^{\# \sigma}  \sum_{G \in \G_n: G_{1\bullet}=k+1} \beta^{\# G} \textbf{1}_{\{\sigma \subset G\}} \\
&= \frac{1}{z_n(\alpha,\beta)} \sum_{\sigma}\alpha^{\# \sigma} \sum_{m=n+k}^{n^2-(n-(k+1))} \beta^m  \sum_{G \in \G_n: G_{1\bullet}=k+1, \# G=m}  \textbf{1}_{\{\sigma \subset G\}}.
\end{align*}
In the last equality, we have partitioned the terms in the sum over all $G_{1\bullet} = k+1$ into groups of graphs with $m$ edges. Clearly any graph $G$ in the event $\{G_{1\bullet}=k+1\}$, for which $\sigma\subset G$, has at least $n+k$ edges ($n$ for the permutation, $k$ for the additional edges in the first row), and at most $n^2-(n-(k+1))$ edges (there must be exactly $n-(k+1)$ zeros in the first row). Continuing on, the above is equal to
\begin{align*}
&= \frac{1}{z_n(\alpha,\beta)} \sum_{\sigma}\alpha^{\# \sigma} \sum_{m=n+k}^{n^2-(n-(k+1))} \beta^m  \cdot \binom{n-1}{k} \cdot \binom{n^2-n-(n-1)}{m-(n+k)}.
\end{align*}
The combinatorial factor counts the number of ways to pick the additional $k$ edges in the first row, and the number of ways to pick the other $m-(n+k)$ edges from the $n^2-n-(n-1)$ possibilities, since $n$ entries are fixed due to the permutation $\sigma$, and the remaining $n-1$ entries of the first row are fixed by the (exactly) $k$ ones in the first row. The product of these two represents the number of graphs $G$ with $m$ total edges with $k+1$ in the first row, such that $\sigma \subset G$.

Factoring out $\beta^{n+k}$ and shifting the summation index, the above becomes
\begin{align*}
&= \frac{1}{z_n(\alpha,\beta)} \sum_{\sigma} \alpha^{\# \sigma} \beta^{n+k} \binom{n-1}{k} \sum_{m=0}^{n^2-(n-(k+1)) - (n+k)} \beta^m \binom{n^2-n-(n-1)}{m} \\
&= \frac{1}{\alpha_{n\uparrow 1} \left(\frac{\beta}{1+\beta} \right)^n (1+\beta)^{n^2}}  \alpha_{n\uparrow 1} \beta^{n+k} \binom{n-1}{k} \underbrace{\sum_{m=0}^{n^2-n-(n-1)} \beta^m \binom{n^2-n-(n-1)}{m}}_{=(1+\beta)^{n^2-n-(n-1)} \text{ by the binomial formula}} \\
&= \binom{n-1}{k} \beta^k (1+\beta)^{n-n^2} (1+\beta)^{n^2-n-(n-1)} \\
&= \binom{n-1}{k} \beta^k (1+\beta)^{-(n-1)}.
\end{align*}
The above is the pmf of a Binom$(n-1,\frac{\beta}{1+\beta})$ variable at $k$.

\end{proof}

\begin{proof}[Proof of Lemma \ref{crpperm}.] The validity of the probabilities
\begin{align*}
P_n(\sigma) = \frac{\alpha^{\#\sigma}}{\alpha_{n\uparrow 1}} \hspace{1em} \text{for }\sigma\in \S_n
\end{align*}
follows from the recursion
\begin{align*}
\sum_{\sigma \in \S_n} \alpha^{\# \sigma} &= \sum_{\sigma \in \S_{n-1}} \left((n-1)\alpha^{\# \sigma} + \alpha \cdot \alpha^{\#\sigma}\right) \\
&= (\alpha+n-1) \sum_{\sigma \in \S_{n-1}} \alpha^{\# \sigma} .
\end{align*}
Letting $\varphi^\text{dr}_n: \{0,1\}^{(n+1)\times(n+1)} \to \{0,1\}^{n\times n}$ denote the delete-and-repair mapping defined by (\ref{dr}), consistency amounts to showing
\begin{align*}
P_n(\sigma) = \sum_{\sigma' \in \S_{n+1}: \varphi^\text{dr}_n(\sigma') = \sigma} P_{n+1}(\sigma').
\end{align*}
The right hand side is a sum over the $n+1$ permutations $\sigma' \in \S_{n+1}$ that delete and repair down to $\sigma$. Exactly one of these permutations has $\# \sigma' = \#\sigma + 1$, namely the permutation obtained by placing $n+1$ in its own cycle. The other $\sigma'$ have the same number of cycles as $\sigma$.
\begin{align*}
\sum_{\sigma' \in \S_{n+1}: \varphi^\text{dr}_n(\sigma') = \sigma} P_{n+1}(\sigma') &= \frac{1}{\alpha_{(n+1)\uparrow 1}} \cdot \left(n \alpha^{\# \sigma} + \alpha^{\# \sigma + 1} \right) \\
&= \alpha^{\# \sigma} \cdot \frac{\alpha+n}{\alpha_{(n+1)\uparrow 1}} \\
&= \frac{\alpha^{\# \sigma}}{\alpha_{n\uparrow 1}},
\end{align*}
as desired. A similar calculation yields delete and repair consistency for the partitions generated by the CRP$(\alpha)$.

\end{proof}

\begin{proof}[Proof of Proposition \ref{imposs}]
It must be the case that $\alpha_n,\beta_n > 0$ in order for the probabilities defined by 
\begin{align*}
P_n(G) \propto \beta_n^{\# G} \per_{\alpha_n} (G)
\end{align*}
to be valid for all $G \in \G_n=\{0,1\}^{n\times n}$. Keeping this in mind, we compute the right hand side of (\ref{norm}) for $G_1$ and $G_2$.

The set of $G'\in\G_5$ for which $\varphi^\text{dr}_4(G')=G_1$ (and also contain at least one permutation) is
\begin{align*}
(\varphi^\text{dr}_4)^{-1}(G_1) = 
\begin{split} 
\left\{ \begin{bmatrix}
0&1&0&0&0\\1&0&1&0&0\\0&1&0&1&0\\ 1&0&1&0&0\\ *&*&*&*&1 
\end{bmatrix}, 
\begin{bmatrix}
0&1&0&0&*\\1&0&1&0&*\\0&1&0&1&*\\ 1&0&1&0&*\\ 0&0&0&0&1 
\end{bmatrix}, 
\begin{bmatrix}
0&*&0&0&1\\1&0&1&0&0\\0&1&0&1&0\\ 1&0&1&0&0\\ 0&1&0&0&1 
\end{bmatrix}, 
\begin{bmatrix}
0&*&0&0&1\\1&0&1&0&0\\0&*&0&1&1\\ 1&0&1&0&0\\ 0&1&0&0&1 
\end{bmatrix}, 
\right. \\
\left. 
\begin{bmatrix}
0&1&0&0&0\\ *&0&1&0&1\\0&1&0&1&0\\ 1&0&1&0&0\\ 1&0&0&0&1 
\end{bmatrix}, 
\begin{bmatrix}
0&1&0&0&0\\ *&0&*&0&1\\0&1&0&1&0\\ 1&0&1&0&0\\ 1&0&1&0&1 
\end{bmatrix}, 
\begin{bmatrix}
0&1&0&0&0\\ *&0&1&0&1\\0&1&0&1&0\\ *&0&1&0&1\\ 1&0&0&0&1 
\end{bmatrix}, 
\begin{bmatrix}
0&1&0&0&0\\ 1&0&1&0&0\\0&1&0&1&0\\ *&0&*&0&1\\ 1&0&1&0&1 
\end{bmatrix}, 
\right. \\
\left. 
\begin{bmatrix}
0&1&0&0&0\\ 1&0&*&0&1\\0&1&0&1&0\\ 1&0&*&0&1\\ 0&0&1&0&1 
\end{bmatrix}, 
\begin{bmatrix}
0&1&0&0&0\\ *&0&*&0&1\\0&1&0&1&0\\ *&0&*&0&1\\ 1&0&1&0&1 
\end{bmatrix},
\begin{bmatrix}
0&1&0&0&0\\ 1&0&*&0&1\\0&1&0&1&0\\ 1&0&1&0&0\\ 0&0&1&0&1 
\end{bmatrix}, 
\begin{bmatrix}
0&1&0&0&0\\ 1&0&1&0&0\\0&*&0&1&1\\ 1&0&1&0&0\\ 0&1&0&0&1 
\end{bmatrix}, 
\right. \\
\left. 
\begin{bmatrix}
0&1&0&0&0\\ 1&0&1&0&0\\0&*&0&*&1\\ 1&0&1&0&0\\ 0&1&0&1&1 
\end{bmatrix}, 
\begin{bmatrix}
0&1&0&0&0\\ 1&0&1&0&0\\0&1&0&*&1\\ 1&0&1&0&0\\ 0&0&0&1&1 
\end{bmatrix}, 
\begin{bmatrix}
0&1&0&0&0\\ 1&0&1&0&0\\0&1&0&1&0\\ *&0&1&0&1\\ 1&0&0&0&1 
\end{bmatrix},
\begin{bmatrix}
0&1&0&0&0\\ 1&0&1&0&0\\0&1&0&1&0\\ 1&0&*&0&1\\ 0&0&1&0&1 
\end{bmatrix}, 
\right. \\
\left. 
\begin{bmatrix}
0&*&0&0&1\\ 1&0&1&0&0\\0&1&0&1&0\\ 1&0&1&0&0\\ 0&1&0&0&0 
\end{bmatrix}, 
\begin{bmatrix}
0&*&0&0&1\\ 1&0&1&0&0\\0&*&0&1&1\\ 1&0&1&0&0\\ 0&1&0&0&0 
\end{bmatrix}, 
\begin{bmatrix}
0&1&0&0&0\\ *&0&1&0&1\\0&1&0&1&0\\ 1&0&1&0&0\\ 1&0&0&0&0 
\end{bmatrix}, 
\begin{bmatrix}
0&1&0&0&0\\ *&0&*&0&1\\0&1&0&1&0\\ 1&0&1&0&0\\ 1&0&1&0&0 
\end{bmatrix}, 
\right. \\
\left.
\begin{bmatrix}
0&1&0&0&0\\ *&0&1&0&1\\0&1&0&1&0\\ *&0&1&0&1\\ 1&0&0&0&0 
\end{bmatrix}, 
\begin{bmatrix}
0&1&0&0&0\\ 1&0&*&0&1\\0&1&0&1&0\\ 1&0&*&0&1\\ 0&0&1&0&0 
\end{bmatrix}, 
\begin{bmatrix}
0&1&0&0&0\\ 1&0&1&0&0\\0&1&0&1&0\\ *&0&*&0&1\\ 1&0&1&0&0 
\end{bmatrix}, 
\begin{bmatrix}
0&1&0&0&0\\ *&0&*&0&1\\0&1&0&1&0\\ *&0&*&0&1\\ 1&0&1&0&0 
\end{bmatrix},
\right. \\
\left. 
\begin{bmatrix}
0&1&0&0&0\\ 1&0&*&0&1\\0&1&0&1&0\\ 1&0&1&0&0\\ 0&0&1&0&0 
\end{bmatrix},
\begin{bmatrix}
0&1&0&0&0\\ 1&0&1&0&0\\0&*&0&1&1\\ 1&0&1&0&0\\ 0&1&0&0&0 
\end{bmatrix}, 
\begin{bmatrix}
0&1&0&0&0\\ 1&0&1&0&0\\0&*&0&*&1\\ 1&0&1&0&0\\ 0&1&0&1&0 
\end{bmatrix}, 
\begin{bmatrix}
0&1&0&0&0\\ 1&0&1&0&0\\0&1&0&*&1\\ 1&0&1&0&0\\ 0&0&0&1&0 
\end{bmatrix},
\right. \\
\left.
\begin{bmatrix}
0&1&0&0&0\\ 1&0&1&0&0\\0&1&0&1&0\\ *&0&1&0&1\\ 1&0&0&0&0 
\end{bmatrix}, 
\begin{bmatrix}
0&1&0&0&0\\ 1&0&1&0&0\\0&1&0&1&0\\ 1&0&*&0&1\\ 0&0&1&0&0 
\end{bmatrix}
 \right\} 
\end{split}
\end{align*}
where $*$ means either 0 or 1. In total there are 139 graphs in $\G_5$ which delete and repair down to $G_1$ (that contain at least one permutation). Listed below are the permutations in $\S_5$ that project down to $(1234)$ or $(12)(34)$,
\begin{align*}
\{(12345),(12354),(12534),(15234),(1234)(5),(125)(34),(152)(34),(12)(345),(12)(354),(12)(34)(5)\}.
\end{align*}
Going through each $\sigma'$ in the above set in the order listed above and computing the sum
\begin{align*}
\sum_{G' \in (\varphi^\text{dr}_4)^{-1}(G_1)} \beta^{\# G'} \textbf{1}_{\{\sigma' \subset G'\}},
\end{align*}
will show that in this case, the right hand side of (\ref{norm}) becomes
\begin{align*}
&= \frac{1}{\beta_4^7 \cdot (\alpha_4+\alpha_4^2)}\bigg[ \alpha_5(\beta_5^{9}+2\beta_5^{10}+\beta_5^{11}+\beta_5^9+2\beta_5^{10}+\beta_5^{11}+\beta_5^9+3\beta_5^{10}+3\beta_5^{11}+\beta_5^{12}+\beta_5^{9}+\beta_5^{10}+\beta_5^8\\
&+2\beta_5^9+\beta_5^{10}+\beta_5^8+2\beta_5^9+\beta_5^{10}+\beta_5^8+3\beta_5^9+3\beta_5^{10}+\beta_5^{11}+\beta_5^8+\beta_5^9) + \alpha_5(\beta_5^9+2\beta_5^{10}+\beta_5^{11}+\beta_5^9+\beta_5^{10}\\
&+\beta_5^8+2\beta_5^9+\beta_5^{10}+\beta_5^8+\beta_5^9)+\alpha_5(\beta_5^9+2\beta_5^{10}+\beta_5^{11}+\beta_5^9+2\beta_5^{10}+\beta_5^{11}+\beta_5^9+3\beta_5^{10}+3\beta_5^{11}+\beta_5^{12}+\beta_5^9\\
&+\beta_5^{10}+\beta_5^8+2\beta_5^9+\beta_5^{10}+\beta_5^8+2\beta_5^9+\beta_5^{10}+\beta_5^8+3\beta_5^9+3\beta_5^{10}+\beta_5^{11}+\beta_5^8+\beta_5^9)+\alpha_5(\beta_5^9+\beta_5^{10}+\beta_5^9\\
&+2\beta_5^{10}+\beta_5^{11}+\beta_5^8+\beta_5^9+\beta_5^8+2\beta_5^{9}+\beta_5^{10})+\alpha_5^2(\beta_5^8+4\beta_5^9+6\beta_5^{10}+4\beta_5^{11}+\beta_5^{12}+4\beta_5^9+6\beta_5^{10}+4\beta_5^{11}\\
&+\beta_5^{12}+\beta_5^{10}+\beta_5^{10}+\beta_5^{11}+\beta_5^{9}+\beta_5^{10}+\beta_5^{10}+\beta_5^{11}+\beta_5^{10}+\beta_5^{11}+\beta_5^{10}+\beta_5^{11}+\beta_5^{10}+\beta_5^{11}+\beta_5^{10}+2\beta_5^{11}\\
&+\beta_5^{12}+\beta_5^{10}+\beta_5^{9}+\beta_5^{10}+\beta_5^{10}+\beta_5^{11}+\beta_5^{10}+\beta_5^{10}+\beta_5^{9}+\beta_5^{10})+\alpha_5^2(\beta_5^9+\beta_5^{10}+\beta_5^{9}+2\beta_5^{10}+\beta_5^{11}+\beta_5^{9}\\
&+2\beta_5^{10}+\beta_5^{11}+\beta_5^{9}+3\beta_5^{10}+3\beta_5^{11}+\beta_5^{12}+\beta_5^{8}+\beta_5^{9}+\beta_5^{8}+2\beta_5^{9}+\beta_5^{10}+\beta_5^{8}+2\beta_5^{9}+\beta_5^{10}+\beta_5^{8}+3\beta_5^{9}\\
&+3\beta_5^{10}+\beta_5^{11})+\alpha_5^2(\beta_5^9+\beta_5^{10}+\beta_5^{9}+2\beta_5^{10}+\beta_5^{11}+\beta_5^{8}+\beta_5^{9}+\beta_5^{8}+2\beta_5^{9}+\beta_5^{10}) +\alpha_5^2(\beta_5^9+2\beta_5^{10}+\beta_5^{11}\\
&+\beta_5^{9}+2\beta_5^{10}+\beta_5^{11}+\beta_5^{9}+3\beta_5^{10}+3\beta_5^{11}+\beta_5^{12}+\beta_5^{9}+\beta_5^{10}+\beta_5^{8}+2\beta_5^{9}+\beta_5^{10}+\beta_5^{8}+2\beta_5^{9}+\beta_5^{10}+\beta_5^{8}\\
&+3\beta_5^{9}+3\beta_5^{10}+\beta_5^{11}+\beta_5^{8}+\beta_5^{9})+\alpha_5^2(\beta_5^9+2\beta_5^{10}+\beta_5^{11}+\beta_5^{9}+\beta_5^{10}+\beta_5^{8}+2\beta_5^{9}+\beta_5^{10}+\beta_5^{8}+\beta_5^{9})\\
&+\alpha_5^3(\beta_5^8+4\beta_5^{9}+6\beta_5^{10}+4\beta_5^{11}+\beta_5^{12}+4\beta_5^{9}+6\beta_5^{10}+4\beta_5^{11}+\beta_5^{12}+\beta_5^{10}+\beta_5^{10}+\beta_5^{11}+\beta_5^{10}+\beta_5^{10}+\beta_5^{11}\\
&+\beta_5^{10}+\beta_5^{11}+\beta_5^{10}+\beta_5^{11}+\beta_5^{10}+\beta_5^{11}+\beta_5^{10}+2\beta_5^{11}+\beta_5^{12}+\beta_5^{9}+\beta_5^{10}+\beta_5^{9}+\beta_5^{10}+\beta_5^{10}+\beta_5^{11}+\beta_5^{10}\\
&+\beta_5^{9}+\beta_5^{10}+\beta_5^{10}) \bigg].
\end{align*}
Simplifying, the above becomes
\begin{align}
\label{g1drn}
&= \frac{1}{\beta_4^7 \cdot (\alpha_4+\alpha_4^2)}\bigg[ \alpha_5 (12\beta_5^8+34\beta_5^9 + 34\beta_5^{10}+14\beta_5^{11}+2\beta_5^{12})+\alpha_5^2(13\beta_5^8+45\beta_5^9+60\beta_5^{10}+30\beta_5^{11}+5\beta_5^{12}) \\
&+ \alpha_5^3(\beta_5^{8}+11\beta_5^9 + 26\beta_5^{10} + 16\beta_5^{11}+3\beta_5^{12}) \bigg]. \nonumber
\end{align}

The set of $G' \in \G_5$ for which $\varphi^\text{dr}_4(G')=G_2$ (and also contain at least one permutation) is
\begin{align*}
(\varphi^\text{dr}_4)^{-1}(G_2) = 
\begin{split} 
\left\{ \begin{bmatrix}
0&1&0&0&0\\0&0&1&0&0\\1&1&0&1&0\\ 1&0&0&1&0\\ *&*&*&*&1 
\end{bmatrix}, 
\begin{bmatrix}
0&1&0&0&*\\0&0&1&0&*\\1&1&0&1&*\\ 1&0&0&1&*\\ 0&0&0&0&1 
\end{bmatrix}, 
\begin{bmatrix}
0&*&0&0&1\\0&0&1&0&0\\1&1&0&1&0\\ 1&0&0&1&0\\ 0&1&0&0&1 
\end{bmatrix}, 
\begin{bmatrix}
0&*&0&0&1\\0&0&1&0&0\\1&*&0&1&1\\ 1&0&0&1&0\\ 0&1&0&0&1 
\end{bmatrix}, 
\right. \\
\left. 
\begin{bmatrix}
0&1&0&0&0\\0&0&*&0&1\\1&1&0&1&0\\ 1&0&0&1&0\\ 0&0&1&0&1 
\end{bmatrix}, 
\begin{bmatrix}
0&1&0&0&0\\0&0&1&0&0\\ *&1&0&1&1\\ 1&0&0&1&0\\ 1&0&0&0&1 
\end{bmatrix}, 
\begin{bmatrix}
0&1&0&0&0\\0&0&1&0&0\\ *&*&0&1&1\\ 1&0&0&1&0\\ 1&1&0&0&1 
\end{bmatrix}, 
\begin{bmatrix}
0&1&0&0&0\\0&0&1&0&0\\ 1&*&0&*&1\\ 1&0&0&1&0\\ 0&1&0&1&1 
\end{bmatrix}, 
\right. \\
\left.
\begin{bmatrix}
0&1&0&0&0\\0&0&1&0&0\\ *&1&0&*&1\\ 1&0&0&1&0\\ 1&0&0&1&1 
\end{bmatrix}, 
\begin{bmatrix}
0&1&0&0&0\\0&0&1&0&0\\ *&*&0&*&1\\ 1&0&0&1&0\\ 1&1&0&1&1 
\end{bmatrix}, 
\begin{bmatrix}
0&1&0&0&0\\0&0&1&0&0\\ *&1&0&1&1\\ *&0&0&1&1\\ 1&0&0&0&1 
\end{bmatrix}, 
\begin{bmatrix}
0&1&0&0&0\\0&0&1&0&0\\ 1&1&0&*&1\\ 1&0&0&*&1\\ 0&0&0&1&1 
\end{bmatrix}, 
\right. \\
\left.
\begin{bmatrix}
0&1&0&0&0\\0&0&1&0&0\\ 1&1&0&1&0\\ *&0&0&*&1\\ 1&0&0&1&1 
\end{bmatrix}, 
\begin{bmatrix}
0&1&0&0&0\\0&0&1&0&0\\ *&1&0&*&1\\ *&0&0&*&1\\ 1&0&0&1&1 
\end{bmatrix}, 
\begin{bmatrix}
0&1&0&0&0\\0&0&1&0&0\\ 1&*&0&1&1\\ 1&0&0&1&0\\ 0&1&0&0&1 
\end{bmatrix}, 
\begin{bmatrix}
0&1&0&0&0\\0&0&1&0&0\\ 1&1&0&*&1\\ 1&0&0&1&0\\ 0&0&0&1&1 
\end{bmatrix}, 
\right. \\
\left.
\begin{bmatrix}
0&1&0&0&0\\0&0&1&0&0\\ 1&1&0&1&0\\ *&0&0&1&1\\ 1&0&0&0&1 
\end{bmatrix}, 
\begin{bmatrix}
0&1&0&0&0\\0&0&1&0&0\\ 1&1&0&1&0\\ 1&0&0&*&1\\ 0&0&0&1&1 
\end{bmatrix}, 
\begin{bmatrix}
0&*&0&0&1\\0&0&1&0&0\\ 1&1&0&1&0\\ 1&0&0&1&0\\ 0&1&0&0&0 
\end{bmatrix},
\begin{bmatrix}
0&*&0&0&1\\0&0&1&0&0\\ 1&*&0&1&1\\ 1&0&0&1&0\\ 0&1&0&0&0 
\end{bmatrix},
\right. \\
\left.
\begin{bmatrix}
0&1&0&0&0\\0&0&*&0&1\\ 1&1&0&1&0\\ 1&0&0&1&0\\ 0&0&1&0&0 
\end{bmatrix},
\begin{bmatrix}
0&1&0&0&0\\0&0&1&0&0\\ *&1&0&1&1\\ 1&0&0&1&0\\ 1&0&0&0&0 
\end{bmatrix},
\begin{bmatrix}
0&1&0&0&0\\0&0&1&0&0\\ *&*&0&1&1\\ 1&0&0&1&0\\ 1&1&0&0&0 
\end{bmatrix},
\begin{bmatrix}
0&1&0&0&0\\0&0&1&0&0\\ 1&*&0&*&1\\ 1&0&0&1&0\\ 0&1&0&1&0 
\end{bmatrix},
\right. \\
\left.
\begin{bmatrix}
0&1&0&0&0\\0&0&1&0&0\\ *&1&0&*&1\\ 1&0&0&1&0\\ 1&0&0&1&0 
\end{bmatrix},
\begin{bmatrix}
0&1&0&0&0\\0&0&1&0&0\\ *&*&0&*&1\\ 1&0&0&1&0\\ 1&1&0&1&0 
\end{bmatrix},
\begin{bmatrix}
0&1&0&0&0\\0&0&1&0&0\\ *&1&0&1&1\\ *&0&0&1&1\\ 1&0&0&0&0 
\end{bmatrix},
\begin{bmatrix}
0&1&0&0&0\\0&0&1&0&0\\ 1&1&0&1&0\\ *&0&0&*&1\\ 1&0&0&1&0 
\end{bmatrix},
\right. \\
\left.
\begin{bmatrix}
0&1&0&0&0\\0&0&1&0&0\\ 1&1&0&*&1\\ 1&0&0&*&1\\ 0&0&0&1&0 
\end{bmatrix},
\begin{bmatrix}
0&1&0&0&0\\0&0&1&0&0\\ *&1&0&*&1\\ *&0&0&*&1\\ 1&0&0&1&0 
\end{bmatrix},
\begin{bmatrix}
0&1&0&0&0\\0&0&1&0&0\\ 1&*&0&1&1\\ 1&0&0&1&0\\ 0&1&0&0&0 
\end{bmatrix},
\begin{bmatrix}
0&1&0&0&0\\0&0&1&0&0\\ 1&1&0&*&1\\ 1&0&0&1&0\\ 0&0&0&1&0 
\end{bmatrix},
\right. \\
\left.
\begin{bmatrix}
0&1&0&0&0\\0&0&1&0&0\\ 1&1&0&1&0\\ *&0&0&1&1\\ 1&0&0&0&0 
\end{bmatrix},
\begin{bmatrix}
0&1&0&0&0\\0&0&1&0&0\\ 1&1&0&1&0\\ 1&0&0&*&1\\ 0&0&0&1&0 
\end{bmatrix}
 \right\} 
\end{split}
\end{align*}
where $*$ means either 0 or 1. In total there are 163 graphs in $\G_5$ which delete and repair down to $G_2$ (that contain at least one permutation). Listed below are the permutations in $\S_5$ that project down to $(1234)$ or $(123)(4)$,
\begin{align*}
\{(12345),(12354),(12534),(15234),(1234)(5),(1235)(4),(1253)(4),(1523)(4),(123)(45),(123)(4)(5)\}.
\end{align*}
Going through each $\sigma'$ in the above set in the order listed above and computing the sum
\begin{align*}
\sum_{G' \in (\varphi^\text{dr}_4)^{-1}(G_1)} \beta^{\# G'} \textbf{1}_{\{\sigma' \subset G'\}},
\end{align*}
will show that in this case, the right hand side of (\ref{norm}) becomes
\begin{align*}
&= \frac{1}{\beta_4^7 \cdot (\alpha_4+\alpha_4^2)}\bigg[ \alpha_5(\beta_5^9+2\beta_5^{10}+\beta_5^{11}+\beta_5^{9}+2\beta_5^{10}+\beta_5^{11}+\beta_5^{9}+3\beta_5^{10}+3\beta_5^{11}+\beta_5^{12}+\beta_5^{9}+\beta_5^{10}+\beta_5^{8}\\
&+2\beta_5^{9}+\beta_5^{10}+\beta_5^{8}+2\beta_5^{9}+\beta_5^{10}+\beta_5^{8}+3\beta_5^{9}+3\beta_5^{10}+\beta_5^{11}+\beta_5^{8}+\beta_5^{9})+\alpha_5(\beta_5^9+2\beta_5^{10}+\beta_5^{11}+\beta_5^{9}+2\beta_5^{10}\\
&+\beta_5^{11}+\beta_5^{9}+3\beta_5^{10}+3\beta_5^{11}+\beta_5^{12}+\beta_5^{9}+2\beta_5^{10}+\beta_5^{11}+\beta_5^{9}+3\beta_5^{10}+3\beta_5^{11}+\beta_5^{12}+\beta_5^{9}+\beta_5^{10}+\beta_5^{8}+2\beta_5^{9}\\
&+\beta_5^{10}+\beta_5^{8}+2\beta_5^{9}+\beta_5^{10}+\beta_5^{8}+3\beta_5^{9}+3\beta_5^{10}+\beta_5^{11}+\beta_5^{8}+2\beta_5^{9}+\beta_5^{10}+\beta_5^{8}+3\beta_5^{9}+3\beta_5^{10}+\beta_5^{11}+\beta_5^{8}+\beta_5^{9})\\
&+\alpha_5(\beta_5^9+\beta_5^{10}+\beta_5^{8}+\beta_5^{9})+\alpha_5(\beta_5^9+\beta_5^{10}+\beta_5^{9}+2\beta_5^{10}+\beta_5^{11}+\beta_5^{8}+\beta_5^{9}+\beta_5^{8}+2\beta_5^{9}+\beta_5^{10})+\alpha_5^2(\beta_5^8+4\beta_5^{9}\\
&+6\beta_5^{10}+4\beta_5^{11}+\beta_5^{12}+4\beta_5^{9}+6\beta_5^{10}+4\beta_5^{11}+\beta_5^{12}+\beta_5^{10}+\beta_5^{10}+\beta_5^{11}+\beta_5^{10}+\beta_5^{9}+\beta_5^{10}+\beta_5^{9}+2\beta_5^{10}+\beta_5^{11}\\
&+\beta_5^{10}+\beta_5^{11}+\beta_5^{10}+\beta_5^{11}+\beta_5^{10}+2\beta_5^{11}+\beta_5^{12}+\beta_5^{10}+\beta_5^{11}+\beta_5^{10}+\beta_5^{11}+\beta_5^{10}+\beta_5^{11}+\beta_5^{10}+2\beta_5^{11}+\beta_5^{12}\\
&+\beta_5^{9}+\beta_5^{10}+\beta_5^{10}+\beta_5^{10}+\beta_5^{9}+\beta_5^{10})+\alpha_5^2(\beta_5^{9}+\beta_5^{10}+\beta_5^{9}+2\beta_5^{10}+\beta_5^{11}+\beta_5^{9}+2\beta_5^{10}+\beta_5^{11}+\beta_5^{9}+3\beta_5^{10}\\
&+3\beta_5^{11}+\beta_5^{12}+\beta_5^{9}+2\beta_5^{10}+\beta_5^{11}+\beta_5^{9}+3\beta_5^{10}+3\beta_5^{11}+\beta_5^{12}+\beta_5^{8}+\beta_5^{9}+\beta_5^{8}+2\beta_5^{9}+\beta_5^{10}+\beta_5^{8}+2\beta_5^{9}+\beta_5^{10}\\
&+\beta_5^{8}+3\beta_5^{9}+3\beta_5^{10}+\beta_5^{11}+\beta_5^{8}+2\beta_5^{9}+\beta_5^{10}+\beta_5^{8}+3\beta_5^{9}+3\beta_5^{10}+\beta_5^{11})+ \alpha_5^2(\beta_5^9+\beta_5^{10}+\beta_5^{8}+\beta_5^{9})\\
&+\alpha_5^2(\beta_5^9+\beta_5^{10}+\beta_5^{9}+2\beta_5^{10}+\beta_5^{11}+\beta_5^{8}+\beta_5^{9}+\beta_5^{8}+2\beta_5^{9}+\beta_5^{10})+\alpha_5^2(\beta_5^9+2\beta_5^{10}+\beta_5^{11}+\beta_5^{9}+2\beta_5^{10}+\beta_5^{11}\\
&+\beta_5^{9}+3\beta_5^{10}+3\beta_5^{11}+\beta_5^{12}+\beta_5^{9}+\beta_5^{10}+\beta_5^{8}+2\beta_5^{9}+\beta_5^{10}+\beta_5^{8}+2\beta_5^{9}+\beta_5^{10}+\beta_5^{8}+3\beta_5^{9}+3\beta_5^{10}+\beta_5^{11}+\beta_5^{8}\\
&+\beta_5^{9})+\alpha_5^3(\beta_5^8+4\beta_5^{9}+6\beta_5^{10}+4\beta_5^{11}+\beta_5^{12}+4\beta_5^{9}+6\beta_5^{10}+4\beta_5^{11}+\beta_5^{12}+\beta_5^{10}+\beta_5^{10}+\beta_5^{11}+\beta_5^{10}+\beta_5^{10}\\
&+\beta_5^{10}+\beta_5^{11}+\beta_5^{9}+2\beta_5^{10}+\beta_5^{11}+\beta_5^{10}+\beta_5^{11}+\beta_5^{10}+2\beta_5^{11}+\beta_5^{12}+\beta_5^{10}+\beta_5^{11}+\beta_5^{10}+\beta_5^{11}+\beta_5^{10}+\beta_5^{11}\\
&+\beta_5^{10}+2\beta_5^{11}+\beta_5^{12}+\beta_5^{9}+\beta_5^{10}+\beta_5^{9}+\beta_5^{10}+\beta_5^{9}+\beta_5^{10}+\beta_5^{10}) \bigg] .
\end{align*}
Simplifying, the above becomes
\begin{align}
\label{g2drn}
&= \frac{1}{\beta_4^7 \cdot (\alpha_4+\alpha_4^2)}\bigg[ \alpha_5 (13\beta_5^8+38\beta_5^9 + 40\beta_5^{10}+18\beta_5^{11}+3\beta_5^{12})+\alpha_5^2(14\beta_5^8+50\beta_5^9+69\beta_5^{10}+37\beta_5^{11}+7\beta_5^{12}) \\
&+ \alpha_5^3(\beta_5^{8}+12\beta_5^9 + 29\beta_5^{10} + 19\beta_5^{11}+4\beta_5^{12}) \bigg]. \nonumber
\end{align}
Setting expressions (\ref{g1drn}) and (\ref{g2drn}) equal to each other, we have
\begin{align*}
0 = \alpha_5(\beta_5^8+4\beta_5^9+6\beta_5^{10}+4\beta_5^{11}+\beta_5^{12})+\alpha_5^2(\beta_5^8+5\beta_5^9+9\beta_5^{10}+7\beta_5^{11}+2\beta_5^{12})+\alpha_5^3(\beta_5^9+3\beta_5^{10}+3\beta_5^{11}+\beta_5^{12}).
\end{align*}
Since $\alpha_5,\beta_5 > 0$, the right hand side of the above equality is greater than 0, a contradiction.

Next, we prove inconsistency with respect to subselection. The equation (\ref{ltp}) can be rearranged as
\begin{align}
\label{newltp}
\frac{z_{n+1}(\alpha_{n+1},\beta_{n+1})}{z_n(\alpha_n,\beta_n)} \beta_n^{\# G} \sum_{\sigma} \alpha_n^{\# \sigma}\textbf{1}\{\sigma \subset G\} = \sum_{\sigma'} \alpha_{n+1}^{\# \sigma'} \sum_{G' : \varphi^{\text
{ss}}_n(G') = G} \beta_{n+1}^{G'} \textbf{1}_{\{\sigma'\subset G'\}}.
\end{align}
Grouping the inner summands on the right hand side according to the number of edges in $G'$, the right hand side is equal to
\begin{align*}
&= \sum_{\sigma'}\alpha_{n+1}^{\# \sigma'} \sum_{m=\# G+1}^{\# G+2n+1} \beta_{n+1}^m \sum_{G' : \varphi^{\text{ss}}_n(G')=G, \# G' = m} \textbf{1}_{\{\sigma'\subset G'\}},
\end{align*}
since any $G'$ for which $\varphi^{\text{ss}}_n(G')=G$ has at most $\# G + 2n+1$ edges (if node $n+1$ is connected to itself and all other nodes), and at least $\# G + 1$ edges for some $\sigma' \in \S_{n+1}$ to be contained in $G'$. The above can be split into two sums,
\begin{align*}
&= \sum_{\sigma': \sigma'(n+1)=n+1}\alpha_{n+1}^{\# \sigma'} \sum_{m=\# G+1}^{\# G+2n+1} \beta_{n+1}^m \binom{2n}{m-(\# G+1)} + \sum_{\sigma': \sigma'(n+1)\neq n+1}\alpha_{n+1}^{\# \sigma'} \sum_{m=\# G+2}^{\# G+2n+1} \beta_{n+1}^m \binom{2n-1}{m-(\# G+2)},
\end{align*}
since when $\sigma' (n+1) = n+1$ and $\sigma'\subset G'$, there are $2n$ unconstrained entries in $G'$ which can be either zero or one. When $\sigma'(n+1)\neq n+1$ and $\sigma' \subset G'$, there are $2n-1$ unconstrained entries, because $G'_{\sigma^{-1}(n+1),n+1} = G'_{n+1,\sigma(n+1)}=1$ are fixed. In both cases, the top left $n \times n$ submatrix is constrained to be exactly equal to $G$ in order to have $\varphi^{\text{ss}}_n(G')=G$. By the binomial formula, the above becomes
\begin{align*}
&= \beta_{n+1}^{\# G + 1}(1+\beta_{n+1})^{2n} \sum_{\sigma': \sigma'(n+1)=n+1}\alpha_{n+1}^{\# \sigma'} + \beta_{n+1}^{\# G + 2} (1+\beta_{n+1})^{2n-1}\sum_{\sigma': \sigma'(n+1)\neq n+1}\alpha_{n+1}^{\# \sigma'} \\
&= \beta_{n+1}^{\# G + 1}(1+\beta_{n+1})^{2n} \alpha_{n+1}\left( \alpha_{n+1}\right)_{n\uparrow 1} + \beta_{n+1}^{\# G + 2} (1+\beta_{n+1})^{2n-1}n (\alpha_{n+1})_{n\uparrow 1} ,
\end{align*}
where in the last equality we have used Lemma \ref{crpperm}. Combined with (\ref{newltp}), we have shown
\begin{align}
\label{consting}
\frac{z_{n+1}(\alpha_{n+1},\beta_{n+1})}{z_n(\alpha_n,\beta_n)} \left(\frac{\beta_n}{\beta_{n+1}}\right)^{\# G} \sum_{\sigma}\alpha_n^{\# \sigma} \textbf{1}_{\{\sigma \subset G\}} = \beta_{n+1}(1+\beta_{n+1})^{2n}(\alpha_{n+1})_{n\uparrow 1}\left[ \alpha_{n+1} + \frac{n\beta_{n+1}}{1+\beta_{n+1}}  \right].
\end{align}
In particular, the above equality implies that the function $f : \{0,1\}^{n\times n} \to \R$ defined by
\begin{align*}
f(G) \doteq \left(\frac{\beta_n}{\beta_{n+1}}\right)^{\# G} \sum_{\sigma}\alpha_n^{\# \sigma} \textbf{1}_{\{\sigma \subset G\}},
\end{align*}
is constant in $G$. Taking $G_1 = (12\dots n)$ and $G_2 = (1)(23\dots n)$, the requirement $f(G_1)=f(G_2)$ implies that $\alpha_n = \alpha_n^2$ for every $n$. Since we must have $\alpha_n > 0$ for the probabilities (\ref{per}) to be valid, this implies that $\alpha_n \equiv 1$, so that
\begin{align*}
f(G) = \left(\frac{\beta_n}{\beta_{n+1}}\right)^{\# G} \sum_{\sigma} \textbf{1}_{\{\sigma \subset G\}}.
\end{align*}
Taking $G_1 = (12\dots n)$ and $G_2 = (12\dots n)\vee (n)$, the equation $f(G_1)=f(G_2)$ becomes
\begin{align*}
\left(\frac{\beta_n}{\beta_{n+1}}\right)^{n} = \left(\frac{\beta_n}{\beta_{n+1}}\right)^{n+1} \Rightarrow \beta_n \equiv \beta
\end{align*}
is constant in $n$. Since $\alpha_n \equiv 1$ and $\beta_n \equiv \beta > 0$, (\ref{consting}) becomes
\begin{align*}
\frac{z_{n+1}(1,\beta)}{z_n(1,\beta)} \sum_{\sigma} \textbf{1}_{\{\sigma \subset G\}} = \beta(1+\beta)^{2n}\left[1+\frac{n\beta}{1+\beta} \right],
\end{align*}
for every $G$.
Plugging in the formula for the normalization constant gives
\begin{align*}
\frac{\beta}{1+\beta} (1+\beta)^{2n+1} \sum_{\sigma} \textbf{1}_{\{\sigma \subset G\}} = \beta (1+\beta)^{2n}\frac{\beta(n+1)+1}{1+\beta}.
\end{align*}
Simplifying, the above implies that a necessary condition of subselection consistency is that for every $G \in \{0,1\}^{n \times n}$,
\begin{align*}
\sum_{\sigma} \textbf{1}_{\{\sigma \subset G\}} = 1 + \frac{n\beta}{1+\beta}.
\end{align*}
Since the left hand side is not constant in $G$, we reach a contradiction.

\end{proof}

\end{document}